\documentclass[10pt]{amsart}
\usepackage{tikz}
\usepackage[colorlinks]{hyperref}
\usepackage{mathrsfs}
\hypersetup{
citecolor = {blue}
}
\usepackage{color,graphicx,shortvrb}
\usepackage[latin 1]{inputenc}
\usepackage{amssymb}
\usepackage{float}
\usepackage{cancel}
\usepackage{multicol}
\usepackage[active]{srcltx} 

\usepackage{caption}
\usepackage{subcaption}
\usepackage{graphicx}

\usepackage{enumerate}

\newtheorem{theorem}{Theorem}[section]

\newtheorem{corollary}[theorem]{Corollary}

\newtheorem{definition}[theorem]{Definition}
\newtheorem{lemma}[theorem]{Lemma}

\newtheorem{proposition}[theorem]{Proposition}
\newtheorem{remark}[theorem]{Remark}

\newtheorem*{theorem-non}{Theorem}

\def\RR{{\mathbb{R}}}
\def\NN{{\mathbb{N}}}
\def\11{\textbf{$1$}}

\def\io{\int_\Omega }
\def\ido{\int_{\partial \Omega}}
\def\d{\hbox{\rm div\,}}
\def\sg{\hbox{\rm sign\,}}
\def\z{{\bf z}}
\def\GG{\mathcal G}
\newcommand{\sop}{{\rm supp\,}}

\DeclareGraphicsExtensions{.jpg,.pdf,.png,.eps}

\usepackage{enumerate}

\begin{document}

\keywords{nonlinear elliptic equations, 1--Laplacian operator, Gelfand problem}
\subjclass[2010]{35J75, 35J20, 35J92}

\title[Gelfand type problems]{Gelfand type problems involving the $1$--Laplacian operator}

\thanks{}

\author[A. Molino and S. Segura de Le\'on]{A. Molino and S. Segura de Le\'on}

\address{Alexis Molino
\hfill \break\indent Departamento de Matem\'aticas. Universidad de Almer\'ia,
\hfill\break\indent Ctra. de Sacramento sn. 04120. La Ca\~nada de San Urbano. Almer\'ia, Spain} \email{{\tt amolino@ual.es }}

\address{ Sergio Segura de Le\'on
\hfill\break\indent Departament d'An\`alisi Matem\`atica, Universitat de Val\`encia,
\hfill\break\indent Dr. Moliner 50,
46100 Burjassot, Valencia, Spain} \email{{\tt sergio.segura@uv.es.}}

\date{}

\begin{abstract}
In this paper, the theory of Gelfand problems is adapted to the 1--Laplacian setting. Concretely, we deal with the following problem
\begin{equation*}
  \left\{\begin{array}{cc}
  -\Delta_1u=\lambda f(u) &\hbox{in }\Omega\,;\\[2mm]
  u=0 &\hbox{on }\partial\Omega\,;
 \end{array} \right.
\end{equation*}
where $\Omega\subset\mathbb{R}^N$ ($N\ge1$) is a domain, $\lambda \geq 0$ and $f\>:\>[0,+\infty[\to]0,+\infty[$ is any continuous increasing and unbounded function with $f(0)>0$.

 It is proved the existence of a threshold $\lambda^*=\frac{h(\Omega)}{f(0)}$ (being $h(\Omega)$ the Cheeger constant of $\Omega$) such that there exists no solution when $\lambda>\lambda^*$ and the trivial function is always a solution when $\lambda\le\lambda^*$. The radial case is analyzed in more detail showing the existence of multiple solutions (even singular) as well as the behaviour of solutions to problems involving the $p$--Laplacian as $p$ tends to 1, which allows us to identify proper solutions through an extra condition.
\end{abstract}

\maketitle

 \thispagestyle{empty}

\section{Introduction}

This paper is devoted to analyze Gelfand--type problems when the Laplacian operator is replaced with the 1--Laplacian. Regarding the domain $\Omega\subset\mathbb{R}^N$ ($N\ge1$), it is a  bounded open set having Lipschitz--continuous boundary. Our aim is twofold. On the one hand, we obtain solutions to this kind of problems and check that the main properties of Gelfand--problems driven by the $p$--Laplacian (with $p>1$) still hold. We point out that assumptions of great generality on the function that appears on the right hand side are considered. On the other hand, we provide asymptotic information of Gelfand $p$--Laplacian problems as $p$ goes to 1.

Classical Gelfand problem means the existence and  boundedness of positive solutions to the following semilinear ellip\-tic equation
\begin{equation}\label{Gproblem}
\left\{
\begin{array}{cc}
-\Delta u=\lambda e^u,& \hbox{ in } \Omega,\\[2mm]
u=0, & \hbox{ on } \partial \Omega,
\end{array}
\right.
\end{equation}
where $\lambda>0$. This problem was introduced in the lecture notes \cite{Ge} for its application to thermal self--ignition problems of a chemically active mixture of gases in a vessel (other applications can be found in \cite{Ch, Joseph, Keller}).

Many authors have analyzed this problem obtaining a threshold $\lambda^*>0$ beyond which there is no solution and such that there exists a \emph{minimal} solution $w_\lambda$ for each $\lambda\in[0, \lambda^*[$. Even more, the family $\{w_\lambda\,:\,0\leq \lambda < \lambda^*  \}$ is increasing in $\lambda$.  It is worth mentioning that we mean  \emph{minimal} solution when it is the smallest of any other positive solution. The multiplicity of solutions in the radial case have also been studied jointly with the associated bifurcation diagram, which depends on the dimension $N$ (see \cite{JL}).

In recent decades, the classical Gelfand problem has been extended in two main directions. On the one hand, the exponential function is replaced with convex positive functions,  nondecreasing and superlinear at $+\infty$ (like the power function $f(u)=(1 + u)^m, \, m>1$). In this general setting, we refer to the pioneering works \cite{BV, CR2, MP} and the recent survey \cite{Cabre}. On the other hand, larger classes of operators are considered; we highlight the fractional Laplacian \cite{RO}, the Laplacian with a quadratic gradient term \cite{ACM, M}, the 1--homogeneous $p$--Laplacian \cite{CMR} and the $k$--Hessian \cite{J, JS}.

  However, the most studied problem is that driven by the $p$--Laplacian ($p>1$) \cite{CS1,CFM,GAP, GAPP,S} which extends the previous problem \eqref{Gproblem}  into a more general framework and reads as follows
\begin{equation}\label{plaplacian}
\left\{
\begin{array}{cc}
-\textrm{div}\left(\displaystyle |\nabla u|^{p-2}\nabla u \right)=\lambda f(u),& \hbox{ in } \Omega,
\\
u=0, & \hbox{ on } \partial \Omega,
\end{array}
\right.
 \tag{$Q_\lambda$}
\end{equation}
under the assumptions:
\begin{equation}\label{hipotesis}
\begin{array}{c}
f:[0,\infty[ \to [0,\infty[ \, \hbox{ is an increasing } \mathcal{C}^1 \hbox{ function }
\\
 \hbox{ with } \, \lim_{s\to \infty} f(s)/s^{p-1}=\infty \hbox{ and }  f(0)>0
\end{array}
 \tag{$H_p$}
\end{equation}
Recall that a (weak) solution to the problem is a function $u\in W_0^{1,p}(\Omega)$ satisfying
\begin{equation*}
\int_{\Omega}|\nabla u|^{p-2}\nabla u \cdot \nabla \varphi =\int_\Omega \lambda \, f(u)\varphi, \quad \hbox{ for all } \varphi \in \mathcal{C}_c^{\infty}(\Omega).
\end{equation*}
Note that the solutions are also superharmonic functions. Therefore, by using the strong maximum principle (see for e.g. \cite{Mo}), solutions to \eqref{plaplacian} are positive in $\Omega$.

  On the existence of solutions, it has been shown in (\cite[Theorem 1.4]{CS1}) the existence of a critical value $\lambda_p^*>0$ such that for each $\lambda<\lambda_p^* $ there exists, $w_{\lambda(p)}$, a minimal  and \emph{regular} solution. In addition,  if $\lambda>\lambda_p^*$ then problem $\eqref{plaplacian}$ admits no \emph{regular} solution.
It should be noted that by \emph{regular} solution we mean that $f(w_{\lambda(p)})\in L^\infty(\Omega)$. It should be remembered that, due to regularity results,  this implies that the solution belongs to $\mathcal{C}^{1,\alpha}(\overline {\Omega})$ (see for e.g. \cite{Li}).

    Regarding  the existence and boundedness of solutions to \eqref{plaplacian} for $\lambda=\lambda_p^*$,  called \emph{extremal solution} and we will denote it by $u_p^*:=\lim_{\lambda\to\lambda_p^*}w_{\lambda(p)}$, there are only partial results. Specifically, for $\Omega=B_1(0)$ (the unit ball with center zero), it has been obtained that extremal solution $u_p^*$ is bounded if $N<\frac{p^2+3p}{p-1}$ (\cite[Theorem 1.3]{CCS}). We stress that, in general domains, the optimal dimension that guarantees the boundedness of $u_p^*$ remains unknown. Nevertheless, some interesting results in the original case $p=2$ and $f$ convex satisfying \eqref{hipotesis} should be mentioned. Indeed, the boundedness of extremal solutions  for dimension $N\leq 3$ is   proved  in \cite{N}, for $N=4$ in \cite{V} and, recently, in \cite{CFRS} is obtained for $5\leq N\leq 9$. Observe that this result is optimal since it is well known that for $N\geq 10$, $\lambda=2(N-2)$ and $f(u)=e^u$, there is the presence of the singular $H_0^1(B_1(0))$ stable weak solution: $u^*=-2\log|x|$.
\medskip

This paper is concerned to the limit problem \eqref{plaplacian} as $p$ goes to 1, namely
\begin{equation}\label{problem}
\left\{
\begin{array}{cc}
-\textrm{div}\left(\displaystyle \frac{Du}{|Du|} \right)=\lambda f(u),& \hbox{ in } \Omega,
\\
\\
u=0, & \hbox{ on } \partial \Omega.
\end{array}
\right.
 \tag{$P_\lambda$}
\end{equation}
Here $\Omega$ is an open bounded set  with Lipschitz boundary and $\lambda$ is a positive parameter.
As for nonlinearity $f$, it satisfies the following hypotheses
\begin{equation}\label{hipotes}
\begin{array}{c}
f:[0,\infty[ \to [0,\infty[ \hbox{ is an increasing and continuous }
\\
 \hbox{ function with } \lim_{s\to \infty}f(s)=\infty \hbox{ and } \, f(0)>0
\end{array}
 \tag{$H$}
\end{equation}
Note that \eqref{hipotes} conditions are more relaxed than  \eqref{hipotesis} conditions for $p$--Laplacian problem \eqref{plaplacian}.

 Regarding the 1--Laplacian $\Delta_1=\textrm{div}\left(\displaystyle \frac{Du}{|Du|} \right)$, it has been handled in many articles in recent years. This singular operator has specific features starting from the definition of solution (following \cite{ABCM1, ABCM2, D1}). This notion of solution is introduced in Definition \ref{defi} below. One of the main interests for studying equations involving the 1--Laplacian operator comes from the variational approach to image restoration.

Our objective is to analyze problem \eqref{problem} checking if all the features of Gelfand type problems governed by the $p$--Laplacian operator still hold.
Several aspects of this article are worth noting for their unexpected nature. Firstly, we are able to handle with a general continuous increasing function $f$ without requiring any kind of convexity, growth assumption or smoothness. Taking any of such functions, we make an exhaustive study to  the one--dimensional case (Theorem \eqref{unidimensional} and Proposition \eqref{unidimensional2}). Further, for general domains and without restriction about dimension we show the existence of a critical parameter $\lambda^*=\frac{h(\Omega)}{f(0)}$ (being $h(\Omega)$ the Cheeger constant of $\Omega$) such that there are  solutions when $\lambda\leq \lambda^*$ and nonexistence of solution whenever $\lambda > \lambda^*$ (Theorem \eqref{trivial}).
In addition, the minimal solutions correspond to the trivial ones.

   Another unexpected aspect occurs in the radial setting (Section 5) and refers to the bifurcation diagram Figure \ref{fig:2}.
 Being more precise, it is well known that for $p$--Laplacian Gelfand problems \eqref{plaplacian}, in the unit ball with $f(u)=e^u$ and dimensions  $p<N<\frac{p^2+3p}{p-1}$, there exists a critical value $\overline\lambda_p=p^{\, p-1}(N-p)$ for which the problem has countably many bounded radial solutions, see Figure \ref{fig:4} (\cite{GAPP,JS}, see also \cite{Korman}). In our setting, we also find a critical value $\overline \lambda=\frac{N-1}{f(0)}$ for which our problem has a continuum of bounded solutions for every $f$ satisfying \eqref{hipotes} (Theorem \ref{radial}). Nevertheless, just one of them is a limit of $p$--Laplacian type problems. Concretely, we obtain too many solutions and so we wonder which of those are limit of $p$--Laplacian problems. It turns out that we can identify those proper solutions through an extra condition \eqref{clau} (Theorem \ref{limit2}). It seems that most of bounded solutions to $p$--problems tend to unbounded solutions, except for the minimal solutions that tend towards zero (Theorem \ref{limit_p-laplacian}). Thus, from the point of view of bifurcation diagrams, in the $p$--Laplace framework a curve is obtained that oscillates around $\overline\lambda_p$, while in the limit case the diagram has an asymptote in the axis $\lambda=0$. It is really unexpected that bifurcation diagrams corresponding to $p$--problems tend to a bifurcation diagram so close to zero. A further feature is that, for the 1--Laplacian, this diagram does not depend on the dimension $N\geq 2$.

  This paper is organized as follows: In section 2 we recall some properties of the space of functions of bounded variation as well as the concept of solution to problem \eqref{problem}. In section 3 we deal with the one dimensional case. In section 4 we studied the case $N \geq 2$, specifically the existence of a critical value $\lambda^*$ as well as minimal solutions.
Section 5 is devoted to analyze the radial case $\Omega=B_1(0)$ when $N\geq 2$. Finally, in section 6 we discuss the $p$-Laplacian problem. We compare the results obtained with those of the $p$-Laplacian taking limits when p tends to 1.
We end the section by giving estimates of the threshold $\lambda_p^*$ and its limit as $p$ tends to 1 as well as the limit of minimal solutions to $p$--Laplacian problems when $p$ goes to 1.

\section{Preliminaries}

\subsection{Notation}
Throughout this paper, the symbol $\mathcal H^{N-1}(E)$ stands for the $(N - 1)$--dimensional
Hausdorff measure of a set $E\subset\RR^N$ and $|E|$ for its
Lebesgue measure. Moreover, $\Omega\subset \RR^N$ denotes an open bounded set with Lipschitz boundary. Thus, an outward normal unit
  vector $\nu(x)$ is defined for $\mathcal H^{N-1}$--almost every
  $x\in\partial\Omega$.

   We will denote by $W^{1,q}_{0}(\Omega)$ the usual Sobolev space, of measurable functions having weak gradient in $L^{q}(\Omega;\RR^N)$ and zero trace on $\partial \Omega$. Finally, if $1\leq p< N$, we will denote by $\displaystyle p^{*}=Np/(N-p)$ its Sobolev conjugate exponent.

   \subsection{Functions of bounded variation}
   The natural space to study problems involving the 1--Laplacian is the space of functions of bounded variation, defined as
\[
BV(\Omega)=\left\{u\in L^1(\Omega)\,:\, Du \hbox{ is a bounded Radon measure }   \right\}
\]
where $Du:\Omega \to \RR^N$ denotes the distributional gradient of $u$.
In what follows, we denote the distributional gradient by $\nabla u$ if it belongs to $L^1(\Omega;\RR^N)$. We recall that the space $BV(\Omega)$  with norm
\[
\|u\|_{BV(\Omega)}=\io |Du| +\io |u|
\]
is  a Banach space which is non reflexive and non separable.

On the other hand, the notion of a trace on the boundary of functions belonging to Sobolev spaces can be extended to functions $u\in BV(\Omega)$, so that we may write $u\big|_{\partial\Omega}$, through a bounded operator $BV(\Omega)\hookrightarrow L^1(\partial\Omega)$, which is also onto. As a consequence, an equivalent norm on $BV(\Omega)$ can be defined (see \cite{AFP}):
\begin{equation*}
\|u\|=\io |Du| + \ido |u|\, d\mathcal{H}^{N-1}.
\end{equation*}
We will often use this norm in what follows.

We denote by $J_u$ the set of all approximate jump points of $u$.
For every $x\in J_u$ there exist two real numbers $u^+(x)>u^-(x) $ which are the one--sided limits of $u$ at $x$.

In addition, the following continuous embeddings hold
\[
BV(\Omega) \hookrightarrow L^{m}(\Omega)\,,\quad\hbox{for every }1\le m\le\frac{N}{N-1}\,,
\]
which are compact for $1\leq m <\frac{N}{N-1}$. The continuous embedding $BV(\Omega) \hookrightarrow L^{\frac N{N-1}}(\Omega)$
can be improved in the setting of Lorentz space: $BV(\Omega) \hookrightarrow L^{\frac N{N-1},1}(\Omega)$
(we refer to \cite{A} for the embedding $W^{1,1}(\Omega) \hookrightarrow L^{\frac N{N-1},1}(\Omega)$, we remark that the extension to $BV(\Omega)$ is standard: see \cite{Z}). For a detailed account on Lorentz spaces, we refer to \cite{H}. Besides this embedding, we will just need that $L^{\frac N{N-1},1}(\Omega)$ is a Banach space whose dual is the Marcinkiewicz (or weak Lebesgue) space $L^{N,\infty}(\Omega)$. This is the space of all measurable functions $u\>:\>\Omega\to \RR$ satisfying
\[k^N|\{|u|>k\}|\le C\qquad \hbox{for all }k>0\,,\]
where $C$ is a constant independent of $k$. It is straightforward that $L^N(\Omega)\subset L^{N,\infty}(\Omega)$. The simplest instance of a function $L^{N,\infty}(\Omega)\backslash L^N(\Omega)$ is defined by $\displaystyle u(x)=\frac1{|x|}$.

  In this paper, we will use some functionals  which are lower semicontinuous with respect to the $L^1$--convergence. Besides the BV--norm, we also apply the lower semicontinuity of the functional given by
\[
u \mapsto \int_{\Omega} \varphi \, |Du|,
\]
where $\varphi$ is a nonnegative smooth function.

For further properties of functions of bounded variations, we refer to \cite{AFP} (see also \cite{EG,Z}).

\subsection{$L^\infty$--divergence--measure vector fields}

Following \cite{ABCM1,D1}, we define the concept of solution to problem \eqref{problem} through a vector field $\z$ which plays the role of $\frac{Du}{|Du|}$.
Since we need to give a meaning to the dot product of $\z$ and the gradient of a function of bounded variation as well as the weak trace on $\partial \Omega$ of the normal component of $\z$, the Anzellotti theory is required. We remark that most of solutions we deal with are bounded, so that this theory applies. Nonetheless, we find some unbounded solutions which cannot be studied within Anzellotti's theory. Therefore, we have to develop a slight extension of this theory, which we next introduce.

Consider the space $\mathcal X(\Omega)=\left\{{\bf z}\in L^\infty(\Omega;\RR^N) \, : \, \d \,{\bf z}\in L^{N,\infty}(\Omega)\right \}$. For $\z\in \mathcal X(\Omega)$ and $u\in BV(\Omega)$, we define
$
({\bf z},Du):\mathcal{C}_c^\infty(\Omega)\to \RR
$
a distribution  by:
\begin{equation}\label{green-anze}
\left<({\bf z}, Du), \varphi \right>=-\io u\, \varphi \, \d \, {\bf z}-\io u\, {\bf z} \, \nabla \varphi, \quad \forall \, \varphi \in \mathcal{C}_c^\infty(\Omega) \,.
\end{equation}
We point out that the first integral is well--defined since $u\in BV(\Omega)\subset L^{\frac N{N-1},1}(\Omega)$ and $\d\z\in L^{N,\infty}(\Omega)$.
This distribution was introduced in \cite{An} for some pairs $(\z, u)$ satisfying certain compatibility conditions. For instance, $\d\z\in L^N(\Omega)$ and $u\in BV(\Omega)$ or $\d\z\in L^1(\Omega)$ and $u\in BV(\Omega)\cap L^\infty(\Omega)$. In such cases, it is proved that $({\bf z}, Du)$ is a Radon measure with finite total variation. More precisely, it is seen that for every Borel $B$ set with $B\subseteq U\subseteq \Omega$ ($U$ open) it holds
\begin{equation}\label{Borel}
\left| \int_B ({\bf z}, Du) \right| \leq \int_B \left| ({\bf z}, Du)  \right | \leq \|\z\|_{L^\infty(U)}\int_B |Du|\,.
\end{equation}
Taking advantage of the cases already treated, first using truncations $T_k(u)$ and then letting $k$ go to $\infty$, this inequality can easily been extended to every $u\in BV(\Omega)$ and $\z\in\mathcal X(\Omega)$.

We recall the notion of weak trace on $\partial \Omega$ of the normal component of ${\bf z}$, denoted by $\left[{\bf z}, \nu \right]$, where $\nu$ stands for the outer normal unitary vector of $\partial \Omega$. It is defined in \cite{An} as an extension of the classical one, that is,
\begin{equation}\label{strip}
\left[{\bf z}, \nu \right]=\z \cdot \nu,\quad \hbox{for }\, \z \in \mathcal{C}^1(\overline \Omega_\delta; \RR^N)\,,
\end{equation}
where $\Omega_\delta=\left\{x\in \Omega \,:\, {\rm dist}(x,\partial \Omega)<\delta   \right\}$, for some $\delta>0$ sufficiently small.
It satisfies $\left[\z, \nu \right]\in L^\infty(\partial \Omega)$ and $\|\left[{\bf z}, \nu \right]\|_{L^\infty(\partial \Omega)} \leq \| \z \|_{L^\infty(\Omega;\RR^N)}$.

In \cite{An} a Green formula involving the measure $\left(\z, Du   \right)$ and the weak trace $\left[{\bf z}, \nu \right]$ is established, namely:
\begin{equation}\label{Green-2}
\io \left(\z,Du   \right)+ \io u\, \d \, \z =\ido u \left[\z,\nu  \right]d\mathcal{H}^{N-1}
\end{equation}
for those pairs $(\z, u)$ considered in \cite{An}.
This formula also holds for $\z \in \mathcal X(\Omega)$ and $u \in BV(\Omega)$; to prove it, just use truncations again.

\subsection{Definition of solution}

Once we have the suitable theory of $L^\infty$--divergence--measure vector fields, we are in a position to introduce the definition of solution to our problem.

\begin{definition}\label{defi}
A function $u\in BV(\Omega)$ is said to be a solution to problem \eqref{problem} if $f(u)\in L^{N,\infty}(\Omega)$ and there exists a vector field $\z \in L^\infty(\Omega;\RR^N)$ satisfying
\begin{enumerate}
  \item $\|\z\|_\infty\le 1$
  \item $-\d\z=\lambda f(u)$ in $\mathcal D'(\Omega)$
  \item $(\z,Du)=|Du|$ as measures in $\Omega$
  \item $[\z,\nu]\in \sg(-u)$ \quad $\mathcal H^{N-1}$--a.e. on $\partial\Omega$
\end{enumerate}
\end{definition}

We remark that since $\d\z\in L^{N,\infty}(\Omega)$, it follows that the theory of the previous subsection applies. We point out that the unbounded radial solutions we find always satisfy $f(u)=\frac{N-1}{\lambda |x|}$, so that $f(u)\in L^{N,\infty}(\Omega)$.

Notice that, in definition \ref{defi}, the fields $\z$ plays the role of $\displaystyle \frac{Du}{|Du|}$ owing to $\|\z\|_\infty\le1$ and $(\z, Du)=|Du|$ hold.

We point out that, in general, the boundary condition does not hold in the sense of traces. Condition (4) is a weak form of the boundary condition.

\section{Unidimensional case}

In this section, we consider a bounded open set $\Omega\subset\RR$. Then $\Omega$ can be expressed as a union of countably many pairwise disjoint open intervals, that is: $\Omega=\cup_{n\in I}]a_n,b_n[$ with $]a_n, b_n[$ pairwise disjoint,
here $I$ denotes either the set of all positive integers $\NN$ or $\{ 1, 2, \dots , k\}$. Moreover, we write
$L=\max_{n\in I}(b_n-a_n)$, so that there is $n\in I$ which attains this length: $L=b_n-a_n$ (observe that it needs not be unique).

Therefore, problem \eqref{problem} becomes
\begin{equation}\label{uni-d}
  \left\{\begin{array}{cc}
  -\left(\frac{u^\prime}{|u^\prime|}\right)^\prime=\lambda f(u) &\hbox{in }\Omega\,;\\
  u=0 &\hbox{on }\partial\Omega\,.
 \end{array} \right.
\end{equation}

To begin with a remark is in order. Every solution to problem \eqref{uni-d} is constant on each interval $]a_n, b_n[$. Indeed, assume to get a contradiction that a solution $u$ is increasing on an interval $]\alpha,\beta[\subset\Omega$. Since the derivative $u^\prime$ is positive on this interval, we have $\z u^\prime=u^\prime$ and so $\z=1$ on $]\alpha,\beta[$. Hence, $\z^\prime=0$ and  the equation implies $\lambda f(u)=0$, which is not possible. A similar argument holds for a solution which is decreasing on an interval.

\begin{theorem}\label{unidimensional}
Set $\lambda^*=\frac2{Lf(0)}$.
\begin{enumerate}
  \item If $\lambda>\lambda^*$, then problem \eqref{uni-d} has no solution.
  \item If $0<\lambda\le\lambda^*$, then the trivial solution is the minimal solution to problem \eqref{uni-d}.
\end{enumerate}
\end{theorem}

\begin{proof}
(1) Assume that $u$ is a solution to problem \eqref{uni-d}. Then $u$ must be constant on each interval $]a_n,b_n[$, say $u(x)=A_n$. Indeed, from equation \eqref{uni-d} it follows that $\z^\prime(x)=-\lambda f(A_n)$,
so $\z(x)=-\lambda f(A_n)x+C$ for some constant $C$. Having in mind the condition $\|\z\|_\infty\le1$, the steepest slope occurs when $\z(x)=\frac{1}{b_n-a_n}(-2x+a_n+b_n)$. Then $\lambda f(A_n)\le \frac2{b_n-a_n}$ wherewith
\[
f(0)\le f(A_n)\le \frac2{(b_n-a_n)\lambda}\,.
\]
Therefore, $\lambda\le  \frac{2}{(b_n-a_n)f(0)}$ for all $n\in\NN$ and so $\lambda\le  \frac2{Lf(0)}$.

(2) It is straightforward that $u=0$ is a solution with an associated function given by $\z(x)=\lambda f(0)(-x+\frac{b_n+a_n}2)$ for $x\in]a_n,b_n[$.
\end{proof}

\begin{proposition}\label{unidimensional2}
\
\begin{enumerate}
  \item If  $\, 0<\lambda<\lambda^*$, then there exist non trivial nonnegative solutions.
  \item If $\, \lambda=\lambda^*$, then every solution vanishes on each interval whose length is $L$.
\end{enumerate}
\end{proposition}

\begin{proof}
(1) Fix $0<\lambda<\lambda^*$ and split the index set $I=I_1\cup I_2$, with $I_1\ne\emptyset$. Then consider
\begin{equation*}
  u(x)=\left\{\begin{array}{cc}
 \displaystyle   f^{-1}\Big(\frac2{(b_n-a_n)\lambda}\Big) &\hbox{if }x\in]a_n,b_n[ \hbox{ and } n\in I_1\,;\\[5mm]
  0 &\hbox{otherwise;}
  \end{array}\right.
\end{equation*}
 and
 \begin{equation*}
  \z(x)=\left\{\begin{array}{cc}
 \displaystyle   \frac1{b_n-a_n}(-2x+a_n+b_n) &\hbox{if }x\in]a_n,b_n[ \hbox{ and } n\in I_1\,;\\[5mm]
 \displaystyle  \lambda f(0)\left(-x+\frac{b_n+a_n}2\right) &\hbox{otherwise.}
  \end{array}\right.
\end{equation*}
It is easy to check that $u$ is a solution to problem \eqref{uni-d} with associated function $\z$.

Therefore, each choice of $I_1$ generates a nontrivial solution.

(2) Consider an interval $]a_n,b_n[$ such that $b_n-a_n=L$. Since the function $\z$ having the steepest slope is given by $\frac{1}{b_n-a_n}(-2x+a_n+b_n)$, it follows that
$$\lambda^*f(u)=-\z^\prime\le \frac{2}{b_n-a_n}=\frac 2L=\lambda^*f(0)\,.$$
Thus, $f(u)\le f(0)$ and so $u=0$ by the increasing hypothesis on $f$.
\end{proof}

\begin{remark}\rm
Observe that, when $I_1$ is an infinite set and $f(s)=e^s$, the solution defined in the above Proposition is not bounded since $u(x)=\log\Big(\frac2{(b_n-a_n)\lambda}\Big)$ and $b_n-a_n$ is arbitrarily small.
\end{remark}

We next apply the previous results to the case $\Omega=]-1,1[$.

\begin{corollary}\label{corollary}
Let $\Omega=]-1,1[$. Then $\lambda^*=\frac1{f(0)}$ and the solutions to problem \eqref{uni-d} are the following
\begin{enumerate}
  \item (Minimal solutions) If $0<\lambda\le\lambda^*$, the minimal solution is the trivial one.
  \item (Positive solutions) If $0<\lambda<\lambda^*$, a second solution is given by
  \begin{equation*}
  u(x)=f^{-1}\Big(\frac1{\lambda}\Big)
\end{equation*}
 with associated function
$\displaystyle   \z(x)=-x$.
\end{enumerate}
\end{corollary}
This last result is shown in the following Figure \ref{fig:1}. There, the continuum of solutions is illustrated for the Gelfand problem with the 1-Laplacian operator in the unit ball with  dimension one.
\begin{figure}[H]
\centering
\begin{tikzpicture}[scale=0.50]
\draw[<->] (0,8)--(0,0)--(8,0);
\fill (0.1,8) circle (0.1pt) node[above] {$\|u\|_\infty$};
\fill (8,0) circle (0.1pt) node[right] {$\lambda$};
\fill (0,0) circle (2pt) node[below] {$0$};
\fill (6,0) circle (2pt) node[below] {$\lambda^*$};
\node at (5.5,6) {$N=1$};
   \draw[very thick,samples=1000,domain=0.0:6.0, red]
   plot(\x,0);
   \draw[very thick,samples=1000,domain=0.25:6.0, red]
   plot(\x,2/\x-1/3);
\end{tikzpicture}
\caption{}
\label{fig:1}
\end{figure}

\section{Existence and nonexistence of minimal solutions}

Let $N\ge2$ and consider $\Omega\subset\RR^N$ a bounded open set having Lipschitz continuous boundary.
In this section a critical value is found, $\lambda^*>0$ such that a minimal solution to problem \eqref{problem} exists if $\lambda\le\lambda^*$ and no solution exists if $\lambda>\lambda^*$.
This result is proved as a consequence of a criterion to show when the 1--Laplacian equation is solvable (see \cite[Theorem 4.2]{MST1}). In fact, it has the following straightforward consequence.

\begin{proposition}\label{criterio}
   Consider the problem
   \begin{equation}\label{problema-1}
     \left\{\begin{array}{cc}
  -\Delta_1u=\mu &\hbox{in }\Omega\,;\\
  u=0 &\hbox{on }\partial\Omega\,;
 \end{array} \right.
   \end{equation}
   with datum $\mu\in W^{-1,\infty}(\Omega)$.
   \begin{enumerate}
     \item If $\|\mu\|_{W^{-1,\infty}(\Omega)}<1$, then $u \equiv 0$ is a solution
to problem \eqref{problema-1}.
     \item If $\|\mu\|_{W^{-1,\infty}(\Omega)}>1$, then there is not solution to problem \eqref{problema-1}.
   \end{enumerate}
\end{proposition}

A remind is in order: The space $W^{-1,\infty}(\Omega)$ is the dual space of the Sobolev space $W_0^{1,1}(\Omega)$ and its norm is given by
\begin{equation}
\label{norm}  \|\mu\|_{W^{-1,\infty}(\Omega)}
=\sup\left\{\frac{\big|\langle\mu,  v\rangle_{W^{-1,\infty}(\Omega),W_0^{1,1}(\Omega)}\big|}{\int_\Omega|\nabla v|\, dx}\>:\> v\in W_0^{1,1}(\Omega)\backslash\{0\}\right\}\,.
\end{equation}
Applying it to the constant function $\mu \equiv 1$, this expression becomes
\begin{align}\label{norm2}
\nonumber  \|\chi_{\Omega}\|_{W^{-1,\infty}(\Omega)}&=\sup\left\{\frac{\big|\int_\Omega v\, dx\big|}{\int_\Omega|\nabla v|\, dx}\>:\> v\in W_0^{1,1}(\Omega)\backslash\{0\}\right\}\\
  &=\sup\left\{\frac{\int_\Omega |v|\, dx}{\int_\Omega|\nabla v|\, dx}\>:\> v\in W_0^{1,1}(\Omega)\backslash\{0\}\right\}\,.
\end{align}
On the other hand, following the arguments of \cite[Corollary 3.4]{ADS}, it leads to
\begin{equation}\label{norm3}
  \|\chi_{\Omega}\|_{W^{-1,\infty}(\Omega)}
  =\sup\left\{\frac{\int_\Omega |v|\, dx}{\int_\Omega|D v|+\int_{\partial\Omega}|v|\, d\mathcal H^{N-1}}\>:\> v\in BV(\Omega)\backslash\{0\}\right\}\,.
\end{equation}

Our aim is to connect this expression with the Cheeger constant which plays the role of the first eigenvalue of the operator $-\Delta_1$ (see \cite{KF}).
Recall that Cheeger constant of a domain $\Omega$ is defined as
\begin{equation*}
  h(\Omega)=\inf\left\{\frac{P(D)}{|D|}\right\}\,,
\end{equation*}
where the infimum is taken over all nonempty sets of finite perimeter $D\subset\Omega$, and $P(D)$ stands for the perimeter of $D$. As pointed out in \cite{KF}, this constant can be written as
\begin{equation*}
  h(\Omega)=\inf\left\{\frac{\int_\Omega|D v|+\int_{\partial\Omega}|v|\, d\mathcal H^{N-1}}{\int_\Omega |v|\, dx}\>:\> v\in BV(\Omega)\backslash\{0\}\right\}\,.
\end{equation*}
Going back to \eqref{norm3}, it follows that $h(\Omega)=\|\chi_{_{\Omega}}\|_{W^{-1,\infty}(\Omega)}^{-1}$. By the way, as a consequence of \eqref{norm2}, it is straightforward that
\begin{equation*}
  h(\Omega)=\inf\left\{\frac{\int_\Omega|\nabla v|\, dx}{\int_\Omega |v|\, dx}\>:\> v\in W_0^{1,1}(\Omega)\backslash\{0\}\right\}\,.
\end{equation*}

In terms of Cheeger constant, Proposition \ref{criterio} can be written for a constant datum $\lambda$ as follows.
\begin{proposition}\label{criterio2}
   Consider the problem
   \begin{equation}\label{problema-2}
     \left\{\begin{array}{cc}
  -\Delta_1u=\lambda &\hbox{in }\Omega\,;\\
  u=0 &\hbox{on }\partial\Omega\,.
 \end{array} \right.
   \end{equation}
   \begin{enumerate}
     \item If $0<\lambda<h(\Omega)$, then $u \equiv 0$ is a solution
to problem \eqref{problema-2}.
     \item If $\lambda>h(\Omega)$, then there is not solution to problem \eqref{problema-2}.
   \end{enumerate}
\end{proposition}
This result is the key to prove the following theorem on minimal solutions to problem \eqref{problem}.

\begin{theorem}\label{trivial}
Fix $N\geq 2$ and let $\Omega\subset\RR^N$ be a bounded open set having Lipschitz continuous boundary. Set $\lambda^*=\frac{h(\Omega)}{f(0)}$, then
   \begin{enumerate}
     \item If $0<\lambda\le\lambda^*$, then $u \equiv 0$ is a solution
to problem \eqref{problem}.
     \item If $\lambda>\lambda^*$, then there is not solution to problem \eqref{problem}.
   \end{enumerate}
\end{theorem}

\begin{proof}
(1) Consider $0<\lambda<\lambda^*$. Applying Proposition \ref{criterio2} to $\lambda f(0)$, it yields that $u \equiv 0$ is a solution
to problem \eqref{problem}. In order to prove that $u \equiv 0$ is a solution for $\lambda^*$ as well, observe that since it is a solution for any $0<\lambda<\lambda^*$, there exist $\z_\lambda\in L^\infty(\Omega;\RR^N)$ satisfying $\|\z_\lambda\|_\infty\le 1$ and $-\d\z_\lambda=\lambda f(0)$ in the sense of distributions. It follows from $\|\z_\lambda\|_\infty\le 1$ for all $\lambda<\lambda^*$ that there exists a sequence $(\lambda_n)_n$ and $\z_{\lambda^*}\in L^\infty(\Omega;\RR^N)$ such that $\lambda_n\to\lambda^*$ and
\begin{equation*}
  \z_{\lambda_n} \rightharpoonup \z_{\lambda^*}\quad\hbox{*--weakly in }L^\infty(\Omega;\RR^N)\,.
\end{equation*}
Obviously, $\|\z_{\lambda^*}\|_\infty\le 1$. Furthermore, since
\begin{equation*}
  \int_\Omega \z_{\lambda_n}\cdot\nabla v\, dx=\lambda_nf(0)\int_\Omega v\, dx
\end{equation*}
holds for every $v\in W_0^{1,1}(\Omega)$ and every $n\in\NN$, it follows that
\begin{equation*}
  \int_\Omega \z_{\lambda^*}\cdot\nabla v\, dx=\lambda^*f(0)\int_\Omega v\, dx
\end{equation*}
holds for every $v\in W_0^{1,1}(\Omega)$.
Therefore, $-\d\z_{\lambda^*}=\lambda^* f(0)$ holds in the sense of distributions and so $u \equiv 0$ is a solution
to problem \eqref{problem} with $\lambda=\lambda^*$.

(2) Take $\lambda>\lambda^*$ and assume that problem \eqref{problem} has a solution $u$. Then $\lambda f(0)>h(\Omega)$ and it follows that $\|\lambda f(u)\|_{W^{-1,\infty}(\Omega)}>1$. Hence, Proposition \ref{criterio} implies that there is not solution to problem \eqref{problem}, which is a contradiction.
\end{proof}

\begin{remark}\rm
It is worth highlighting that since $u \equiv 0$ is solution for $\lambda^*$, then it is bounded without restriction on dimension $N$ (Theorem \ref{unidimensional} and Theorem \ref{trivial}). This is consistent with the results of (\cite[Theorem 1.3]{CCS}) in which the authors establish that the extremal solution to problem \eqref{plaplacian} is bounded if $N<\frac{p^2+3p}{p-1}$, for every $p>1$.
\end{remark}

\begin{remark}\rm
It is not surprising the trivial function is a minimal solution. After all, a minimal solution to p-Laplacian Gelfand problems is obtained as limit of the sequence recursively defined in $W_0^{1,p}(\Omega)$ by
\[\begin{array}{rcl}
u_0&=&0\\
-\Delta_p u_n&=&\lambda f(u_{n-1}), \quad n\ge1\,.
\end{array}\]
In our case, the unique solution to problem
\[-\Delta_1 u=\lambda f(0)\]
is the trivial solution, wherewith the sequence previously considered also provides the minimal solution.
\end{remark}

\section{The radial case}

This section is devoted to fully describe all the radial solutions to problem \eqref{problem} when the domain is the unit ball, $\Omega=B_1(0)$, of $\RR^N,\, N\geq 2$. It will be shown the existence of bounded (regular) and unbounded (singular) solutions. Recall that the Cheeger constant of the unit ball is $h(\Omega)=N$ (see, for instance, \cite{KF}).

Next it is shown how the vector field $\z$ associated to a solution $u$ must be.
Assume that $u$ is not constant in a radial zone, which we describe as $\{\rho_1<|x|<\rho_2\}$ with $0\le \rho_1<\rho_2\le 1$. Then $u(x)=g(|x|)$ for certain nonconstant smooth function $g$ and consequently $\z(x)=\frac{g^\prime(|x|)\ x}{|g^\prime(|x|)|\,|x|}$. The case $g$ increasing in some interval is not possible since it implies $\z(x)=\frac{x}{|x|}$, so that $\lambda f(u)=-\d\z(x)=-\frac{N-1}{|x|}$ is negative; this fact contradicts the search for nonnegative solutions. Hence, $g$ must be decreasing and so $\z(x)=-\frac{x}{|x|}$ in the whole radial zone $\{\rho_1<|x|<\rho_2\}$.

Assume now that $u$ is constant in a radial zone of $B_1(0)$ (either a smaller ball or a ring). Then $-\d\z=\lambda f(u)$ must be constant and so $\z=-Ax$ for some positive $A$. It is easy to find an estimate on $A$ as a consequence of condition $\|\z\|_\infty\le1$. In the case that $u$ be nontrivial and this zone reaches the boundary, the condition $[\z,\nu]=-1$ must be fulfilled; thus $A=1$ and $\z=-x$.

It is worth remarking that there is only one more setup, as shown in the next result.

        \begin{lemma}\label{tech-re}
        Let $u$ be a solution to problem \eqref{problem} with associated vector field $\z$. If $u$ is constant in a radial zone and nonconstant in another, then there exists $0<\rho<1$ such that $u$ is constant in $B_\rho(0)$ and non constant in $B_1( 0)\backslash B_\rho(0)$.
        Furthermore,
        \begin{equation*}
          \z(x)=\left\{\begin{array}{cc}
         \displaystyle -\frac x\rho\,, &\hbox{if }|x|<\rho\,;\\[3mm]
         \displaystyle  -\frac x{|x|}\,, &\hbox{if }\rho<|x|<1\,.
         \end{array} \right.
        \end{equation*}
        \end{lemma}

        \begin{proof} Assume that there exist $\epsilon>0$ and $\epsilon<\rho<1$ satisfying
        \begin{enumerate}
          \item $u$ is not constant if $\rho-\epsilon<|x|<\rho$
          \item $u$ is constant if $\rho<|x|<\rho+\epsilon$.
        \end{enumerate}
Then
\begin{enumerate}
          \item $\z(x)= -\frac x{|x|}$  if $\rho-\epsilon<|x|<\rho$
          \item $\z(x)=-Ax$  if $\rho<|x|<\rho+\epsilon$.
        \end{enumerate}
        We deduce that $A=\frac1\rho$ to keep $\z$ continuous, avoiding singular measures of $\d\z$. Hence, for $\rho<|x|<\rho+\epsilon$, we have $|\z(x)|=\frac{|x|}\rho>1$, which contradicts the condition $\|\z\|_\infty\le1$. Therefore, the only possible setup is that stated in this Lemma.
        \end{proof}
        \begin{theorem}\label{radial}
Let $N\geq 2$ and set $\lambda^*=\frac{N}{f(0)}$ and $\overline\lambda=\frac{N-1}{f(0)}$. Then,
\begin{enumerate}
     \item For every $0<\lambda<\lambda^*$ there exists a constant nontrivial solution to problem \eqref{problem}.
     \item For every $0<\lambda\le\overline\lambda$  there exists an unbounded solution to problem \eqref{problem}.
     \item For every $0<\lambda\le\overline\lambda$  there exist infinitely many bounded solutions to problem \eqref{problem}. More precisely, for each value $\alpha\in]f^{-1}(\frac N{\lambda }),+\infty[$, we can find a solution satisfying $\|u\|_\infty=\alpha$.
     \item If $\lambda>\overline\lambda$, then every solution to problem \eqref{problem} is constant.
   \end{enumerate}
   Moreover, only the solutions corresponding to $\lambda=\overline\lambda$ satisfy the Dirichlet condition in the sense of traces.
\end{theorem}

\begin{proof}
(1) Fix $0<\lambda<\lambda^*$. Observe that if $u$ is a constant nontrivial solution, then $\z=-x$ wherewith $-\d\z=N$. It follows from $N=\lambda f(u)$ that $\displaystyle u(x)=f^{-1}\left(\frac N\lambda\right)$.

Notice that for $\lambda=\lambda^*$, it follows that $u(x)=f^{-1}\left(f(0)\right)$ recovering the minimal solution.

(2) Fix $0<\lambda\le\overline\lambda$ and let $u$ be a non constant solution. Assume that there is not radial zone where $u$ is constant. Then $\z(x)=-\frac{x}{|x|}$ and $\lambda f(u)=-\d\z(x)=\frac{N-1}{|x|}$.
  Therefore, $u(x)=f^{-1}\left(\frac {N-1}{\lambda |x|}\right)$. It is obvious that $u$ is an unbounded solution.

  (3) Fix $0<\lambda\le\overline\lambda$ and assume that $u$ is a solution which is constant in a radial zone and nonconstant in another. By Lemma \ref{tech-re}, there exists  $0<\rho<1$ such that
        \begin{equation*}
          \z(x)=\left\{\begin{array}{cc}
        \displaystyle   -\frac x\rho\,, &\hbox{if }|x|<\rho\,;\\[3mm]
        \displaystyle   -\frac x{|x|}\,, &\hbox{if }\rho<|x|<1\,.
         \end{array} \right.
        \end{equation*}
        It is now straightforward from this expression that
          \begin{equation*}
          u(x)=\left\{\begin{array}{cc}
         \displaystyle  f^{-1}\left(\frac N{\lambda \rho}\right)\,, &\hbox{if }|x|<\rho\,;\\[3mm]
         \displaystyle  f^{-1}\left(\frac{N-1}{\lambda |x|}\right)\,, &\hbox{if }\rho<|x|<1\,.
         \end{array} \right.
        \end{equation*}
        Observe that $u$ is a bounded solution satisfying $\displaystyle \|u\|_\infty=f^{-1}\left(\frac N{\lambda \rho}\right)$.
        Since $0<\rho<1$, it follows that $\|u\|_\infty$ can take every value of the interval $\displaystyle \Big]f^{-1}\left(\frac N{\lambda}\right),+\infty\Big[$. Notice that $u$ is a discontinuous solution.

  (4) Let $0<\lambda<\lambda^*$. By Lemma \ref{tech-re}, if there exists a non constant solution to problem \eqref{problem}, then there exists $\epsilon>0$ such that $ \z(x)=-\frac x{|x|}$ if $1-\epsilon<|x|<1$. It yields
  $u(x)=f^{-1}\left(\frac{N-1}{\lambda |x|}\right)$ if $1-\epsilon<|x|<1$. Thus, $\frac{N-1}{\lambda |x|}=f(u(x))\ge f(0)$ for all $1-\epsilon<|x|<1$, from where $\frac{N-1}{\lambda}\ge f(0)$ follows.
  Therefore, $\lambda\le \frac{N-1}{f(0)}=\overline\lambda$.

To check the last claim of Theorem \ref{radial}, note that $u(x)=0$ for $|x|=1$ implies $f^{-1}\left(\frac {N-1}{\lambda}\right)=0$ and as a consequence $\lambda=\overline\lambda$.
\end{proof}

The different types of bounded solutions established in Theorem \ref{trivial} and Theorem \ref{radial} can be seen in the following Figure \ref{fig:2}. Note that the blue-colored zone corresponds to the discontinuous solutions with $0 < \lambda \leq \overline\lambda$, of type (3) of the previous theorem. When $\lambda=\overline \lambda=\frac{N-1}{f(0)}$,
the Dirichlet condition in the sense of traces holds: $u(x) = 0$ for $|x|=1$. We will prove in Section 6 that the red-colored continuum corresponds to solutions which are limit of solutions to the corresponding Gelfand problems driven by the $p$--Laplacian.
\begin{figure}[H]
\centering
\begin{tikzpicture}[scale=0.70]
\draw[<->] (0,8)--(0,0)--(8,0);
\fill (0.1,8) circle (0.1pt) node[above] {$\|u\|_\infty$};
\fill (8,0) circle (0.1pt) node[right] {$\lambda$};
\fill (0,0) circle (2pt) node[below] {$0$};
\fill (6,0) circle (2pt) node[below] {$\lambda^*$};
\fill (3,0) circle (2pt) node[below] {$\overline\lambda$};
\node at (5.5,6) {$N\geq 2$};
   \draw[very thick,samples=1000,domain=0.0:6.0, red]
   plot(\x,0);
      \draw[very thick, blue]   (3,0.5)--(3,7.85);
     \draw[very thick,dashed, blue] (2.8,7.85)--(2.8,0.57);
     \draw[very thick,dashed, blue] (2.6,7.85)--(2.6,0.65);
     \draw[very thick,dashed, blue] (2.4,7.85)--(2.4,0.75);
     \draw[very thick,dashed, blue] (2.2,7.85)--(2.2,0.86);
     \draw[very thick,dashed, blue] (2,7.85)--(2,1);
     \draw[very thick,dashed, blue] (1.8,7.85)--(1.8,1.17);
     \draw[very thick,dashed, blue] (1.6,7.85)--(1.6,1.37);
     \draw[very thick,dashed, blue] (1.4,7.85)--(1.4,1.64);
     \draw[very thick,dashed, blue] (1.2,7.85)--(1.2,2);
     \draw[very thick,dashed, blue] (1,7.85)--(1,2.5);
     \draw[very thick,dashed, blue] (0.8,7.85)--(0.8,3.25);
     \draw[very thick,dashed, blue] (0.6,7.85)--(0.6,4.5);
     \draw[very thick,dashed, blue] (0.4,7.85)--(0.4,7);
      \draw[very thick,samples=1000,domain=0.36:6.0, red]
   plot(\x,3/\x-1/2);
       \draw[dashed]  (3,0)--(3,0.5);
\end{tikzpicture}
    \caption{ }
    \label{fig:2}
    \end{figure}

    \begin{remark}\rm
    When $f$ is a bounded function some changes are necessary. One important feature is that no unbounded solution can be found. Consider that the image of $f$ is the interval $[f(0), A[$. Then Theorem \ref{radial} becomes
       \begin{enumerate}
     \item For every $\frac NA<\lambda<\lambda^*$ there exists a constant nontrivial solution to problem \eqref{problem}.
     \item For every $\frac {N-1}A<\lambda\le\overline\lambda$  there exist infinitely many bounded solutions to problem \eqref{problem}. More precisely, for each value $\alpha\in]f^{-1}(\frac N{\lambda }),+\infty[$, we can find a solution satisfying $\|u\|_\infty=\alpha$.
     \item If $\lambda>\overline\lambda$, then every solution to problem \eqref{problem} is constant.
   \end{enumerate}
    \end{remark}

\section{Connection with Gelfand type problems involving the $p$--Laplacian operator}

In this section we are showing when the Gelfand problem for the 1--Laplacian can be seen as the limit of Gelfand problems for the $p$--Laplacian as $p$ goes to 1. Thus, in what follows, we assume \eqref{hipotesis} for all $p$ close enough to 1. We analyze the general case and two aspects of the convergence of radial solutions.

We begin by introducing some notation.
Let $p>1$. We will write $(\lambda, u)\in \GG_p$ if $\lambda>0$ and $u\in W_0^{1,p}(\Omega)$ is a solution to problem \eqref{plaplacian}.

We will write $(\lambda, u)\in \GG$ if $\lambda>0$ and $u\in BV(\Omega)$ and there exist sequences $(p_n)_n$, $(\lambda_n)_n$ and $(u_n)_n$ satisfying
\begin{enumerate}
  \item $p_n\to1$
  \item $\lambda_n\to \lambda$
  \item $u_n(x)\to u(x)$ a.e. in $\Omega$
  \item $(\lambda_n, u_n)\in \GG_{p_n}$ for all $n\in\mathbb N$
\end{enumerate}

Next result shows that $(\lambda, u)\in\GG$ implies that $u$ is a solution to problem \eqref{problem}.

\begin{proposition}\label{limit}
Let $(\lambda_p, u_p)\in \GG_p$ and assume that there exists $g\in L^1(\Omega)$ satisfying $f(u_p)u_p\le g$ for all $p>1$. If $\lambda_p\to \lambda$, then (up to subsequences)
\begin{enumerate}
  \item $u_p\to u$ strongly in $L^1(\Omega)$
  \item $f(u_p)u_p\to f(u)u$ strongly in $L^1(\Omega)$
  \item $f(u_p)\to f(u)$ strongly in $L^1(\Omega)$
  \item $F(u_p)\to F(u)$ strongly in $L^1(\Omega)$
  \item Function $u$ is a solution to problem \eqref{problem}
  \item $\displaystyle \int_\Omega\varphi|Du|=\liminf_{p\to1}\int_\Omega\varphi|\nabla u|^pdx$ for all nonnegative $\varphi\in C_0^\infty(\Omega)$.
\end{enumerate}
\end{proposition}

\begin{proof} Only the case $N\ge2$ will be proved, the case $N=1$ can be handled with minor modifications. We only sketch the proof since it is well--known (see \cite{ABCM2}). Observe that, taking $u_p$ as test function in the $p$--problem \eqref{plaplacian}, the boundedness of the family $\big(f(u_p)u_p\big)_p$ implies that
\[\int_\Omega|\nabla u_p|^p dx\le \lambda_p \int_\Omega g\, dx\le C\]
for some positive constant $C$ which does not depend on $p$. Having in mind that $u_p\big|_{\partial\Omega}=0$ and applying Young's inequality, we deduce that $(u_p)_p$ is bounded in $BV(\Omega)$. Hence, up to a subsequence, \cite[Theorem 3.49]{AFP} yields
\begin{itemize}
  \item $u_p\to u$ strongly in $L^r(\Omega)$ for all $1\le r<\frac N{N-1}$
  \item $u_p(x)\to u(x)$ a.e. in $\Omega$
\end{itemize}
Moreover, as a consequence of our assumption $f(u_p)u_p\le g$, we also obtain
\begin{itemize}
  \item $f(u_p)u_p\to f(u)u$ strongly in $L^1(\Omega)$.
\end{itemize}
Noting that
\[f(u_p)=f(u_p)\chi_{\{u_p\le 1\}}+f(u_p)\chi_{\{u_p> 1\}}\le f(1)+f(u_p)u_p\chi_{\{u_p> 1\}}\le f(1)+g\,,\]
a further consequence is
\begin{itemize}
  \item $f(u_p)\to f(u)$ strongly in $L^1(\Omega)$.
\end{itemize}
Yet another consequence is
\begin{itemize}
  \item $F(u_p)\to F(u)$ strongly in $L^1(\Omega)$
\end{itemize}
which follows from the monotonicity of $f$, in fact:
\[F(s)=\int_0^sf(\sigma)\, d\sigma\le f(s)s\qquad\hbox{for all }s\ge0\,.\]

On the other hand, we may apply the procedure of \cite[Theorem 3.5]{MRST} to obtain a bounded vector field $\z$ such that $\|\z\|_\infty\le1$ and
\begin{itemize}
  \item $|\nabla u_p|^{p-2}\nabla u_p \rightharpoonup \z$ weakly in $L^s(\Omega)$ for all $1\le s<\infty$
\end{itemize}
The above convergences are enough to pass to the limit in the $p$--problems and get
\[-\textrm{div\,}\z=\lambda f(u)\qquad \hbox{in }\mathcal D'(\Omega).\]

In order to check that $(\z, Du)=|Du|$ as measures, fix $\varphi\in C_0^\infty(\Omega)$ such that $\varphi\ge0$ and take $\varphi u_p$ as test function in \eqref{plaplacian}. Then
\[\int_\Omega\varphi|\nabla u_p|^pdx=-\int_\Omega u_p|\nabla u_p|^{p-2}\nabla u_p\cdot\nabla \varphi\, dx+\lambda_p\int_\Omega f(u_p)u_p\varphi\, dx\]
and, by Young's inequality,
\begin{align*}
  \int_\Omega\varphi|\nabla u_p|\, dx & \le \frac1p \int_\Omega\varphi|\nabla u_p|^pdx+\frac{p-1}p\int_\Omega\varphi\, dx \\
  &=-\frac1p \int_\Omega u_p|\nabla u_p|^{p-2}\nabla u_p\cdot\nabla \varphi\, dx+\frac{\lambda_p}p\int_\Omega f(u_p)u_p\varphi\, dx \\
  & \hspace{7,5cm}+\frac{p-1}p\int_\Omega\varphi\, dx\,.
\end{align*}
Owing to the lower semicontinuity on the left hand side and the convergences we have already proven on the right hand side, we may pass to the limit and obtain
\begin{align*}
  \int_\Omega\varphi|D u| & \le\liminf_{p\to1}\int_\Omega\varphi|\nabla u_p|^pdx
  \\
  &=- \int_\Omega u\z \cdot\nabla \varphi\, dx+\lambda\int_\Omega f(u)u\varphi\, dx \\
  &=- \int_\Omega u\z \cdot\nabla \varphi\, dx-\int_\Omega u\,\varphi \, \textrm{div\,}\z\, dx
  \\
  &=\int_\Omega\varphi (\z, Du)\,.
\end{align*}
Since $\int_\Omega\varphi (\z, Du)\le \int_\Omega\varphi|D u|$ always holds, due to $\|\z\|_\infty\le1$, we deduce that the above inequalities become identities. Therefore, $\int_\Omega\varphi|D u|=\int_\Omega\varphi (\z, Du)$ and $\int_\Omega\varphi|D u|=\liminf_{p\to1}\int_\Omega\varphi|\nabla u_p|^pdx$ hold for every $\varphi\in C_0^\infty(\Omega)$.

The boundary condition follows by applying the usual method.
\end{proof}

\begin{remark}\rm
It is straightforward that Proposition \ref{limit} holds if $\lambda_p\to \lambda$ and $(\| u_p \|_\infty)_p$ is bounded.
As a consequence of assertion (1) of Proposition \ref{limit}, we then obtain that $u\in L^\infty(\Omega)$ and
$\| u \|_\infty\le \liminf_{p\to1}\| u_p \|_\infty$.
\end{remark}

\subsection{A necessary condition to obtain the solution as a limit of p-Laplacian type solutions}

From now on, we will focus on the case $N\ge2$ and $\Omega=B_1(0)$, the unit ball, and radial solutions will be analyzed. However, we are include $N=1$ in Theorem \ref{limit2}. Observe that regular solutions are radially decreasing (see \cite{Da}). In this framework, solutions to $\eqref{plaplacian}$ for $p>1$ must satisfy the following quasilinear elliptic equation:
\begin{equation}\label{BVP}
\left\{
\begin{array}{lc}
r^{1-N}\left(r^{N-1}|v'|^{p-2}v'  \right)'+\lambda f(v)=0,& r \hbox{ in } (0,1),
\\[3mm]
v>0, & r \hbox{ in } (0,1),
\\[3mm]
v(1)=0, &
\end{array}
\right.
\end{equation}
being $v(r)=u(|x|)$. We point out that $v\in \mathcal{C}([0,1])$ is a solution of \eqref{BVP} if and only if $v$ is a solution of the integral equation
\begin{equation}\label{gador}
v(r)=\int_r^1 \left[ \frac{\lambda}{t^{N-1}}\int_0^t s^{N-1}f(v(s))ds  \right]^{\frac{1}{p-1}}dt
\end{equation}
which satisfies  $v'(r)<0$ in $(0,1)$ and $v'(0)=0$. Observe that $\|v\|_\infty=v(0):=\alpha>0$.

To solve problem \eqref{BVP}, the following system must be analyzed
\begin{equation}\label{system}
\left\{
\begin{array}{l}
|v'|^{p-2}v'=w,
\\[4mm]
\displaystyle w'=-\frac{N-1}r w-\lambda f(v)
\\[5mm]
v(0)=\alpha, \quad w(0)=0,
\end{array}
\right.
\end{equation}
where $\alpha>0$ is chosen in such a way we get $v(1)=0$. It is then convenient to consider an energy functional:
\begin{equation}\label{energy}
  E(v,w)=\frac1{p'}|w|^{p'}+\lambda F(v),
\end{equation}
whose derivative along trajectories is given by
\[\frac d{dr}E(v,w)=-\frac{N-1}r|w|^{p'}=-\frac{N-1}r|v'|^{p}.\]

\begin{proposition}\label{limit1}
Let $(v,w,\lambda) = (v_p,w_p,\lambda_p)$ be a solution to \eqref{system}, and assume that $\lambda_p\to \lambda_1$ and $(\|v_p\|_\infty)_p$ is bounded.

Then, up to a subsequence, the following properties hold.
\begin{enumerate}
  \item  $(v_p)_p$ converges strongly in $L^s((0,1);r^{N-1}dr)$ to $v_1$, for every $1\le s<\infty$.
  \item  $(w_p)_p$ converges weakly in $L^s((0,1);r^{N-1}dr)$ to $w_1$, for every $1\le s<\infty$.
  Furthermore, $w_1\in L^\infty(0,1)$ with $\|w_1\|_\infty\le1$.
  \item  $(f(v_p)v_p)_p$ converges strongly in $L^s((0,1);r^{N-1}dr)$ to $f(v_1)v_1$, for every $1\le s<\infty$.
  \item  $(f(v_p))_p$ converges strongly in $L^s((0,1);r^{N-1}dr)$ to $f(v_1)$, for every $1\le s<\infty$.
  \item  $(F(v_p))_p$ converges strongly in $L^s((0,1);r^{N-1}dr)$ to $F(v_1)$, for every $1\le s<\infty$.
  \item $v_1\in BV(\sigma,1)$ for every $\sigma>0$.
    \item $w_1$ is Lipschitz--continuous in $(\sigma, 1)$ for every $\sigma>0$.
  \item $-w_1'-\frac{N-1}tw_1=\lambda_1 f(v_1)$ in the sense of distributions.
  \item $|v_1'|=(w_1,v'_1)$ as measures.
  \item The identity
  \begin{equation}\label{clau}
\lambda_1\frac {dF(v_1)}{dr} = -\frac{N-1}r\left|\frac {dv_1}{dr}\right|
\end{equation}
holds in the sense of distributions.
\end{enumerate}
\end{proposition}

\begin{proof}
The proof is similar (but easier) to the proof of \cite[Proposition 16]{SS}, so that we do not provide all details.
The idea for seeing (1)--(5) is to apply Proposition \ref{limit} in a radially symmetric setting and pass to polar coordinates.

\noindent
{\it Proof of 6).}  \quad   \quad

In Proposition \ref{limit} we have got the estimate $\int_{B_1(0)}|\nabla u_p(x)|^p\, dx\le C$, with
$C$ non depending on $p$. Fixed $\sigma>0$, this estimate also holds over
$$
B_1(0)\backslash\overline B_\sigma(0)\,.
$$
Then Young's inequality implies
\begin{align*}
\int_{B_1(0)\backslash\overline B_\sigma(0)}|\nabla u_p(x)|\, dx & \le\frac1p\int_{B_1(0)\backslash\overline B_\sigma(0)}|\nabla u_p(x)|^p\, dx+\frac{p-1}p\big|B_1(0)\big|\\
 & \le C+|B_1(0)|.
\end{align*}
Thus the lower semicontinuity of the total variation yields
\begin{equation*}
  \int_{B_1(0)\backslash\overline B_\sigma(0)}|Du_1|\le \liminf_{p\to\infty}\int_{B_1(0)\backslash\overline B_\sigma(0)}|\nabla u_p|\, dx\le C+|B_1(0)|\,.
\end{equation*}
Passing to polar coordinates, it leads to
\begin{equation*}
\sigma^{N-1} \int_{\sigma}^{1}|v_1'|\le\int_{\sigma}^{1}r^{N-1}|v_1'|\le C'\,,
\end{equation*}
wherewith $v_1$ is a function of bounded variation in $(\sigma,1)$.

\bigskip

\noindent
{\it Proof of 7) and 8).}   \quad

To show that equality 8) holds in the sense of distributions, we choose a test $\psi\in C_0^\infty(0,1)$, fix $0<a<b<1$ in such a way that $\sop\psi\subset(a,b)$ and consider $\varphi$ defined as
\begin{equation}\label{test}
\varphi(x)=\left\{\begin{array}{ll}
\psi({ |x|})\,\frac{1}{|x|^{N-1}}&x\ne0\,;\\
0&x=0\,.
\end{array}\right.
\end{equation}
Having in mind the identity $-\d\z= \lambda f(u)$, we obtain
\begin{align*}
  \lambda \int_{B(0,1)}f(u(x))\varphi(x)\, dx &=\int_{B(0,1)}\z(x)\cdot\nabla \varphi (x)\, dx\\
  &=\int_{B(0,1)}w_1({|x|})\psi'({|x|})\,\frac{dx}{|x|^{N-1}}\\
  &\hspace{2cm}   -\int_{B(0,1)}\frac{N-1}{|x|} w_1({|x|})\psi({|x|})\, \frac{dx}{|x|^{N-1}}.
\end{align*}
Passing to polar coordinates and simplifying, this identity becomes
\begin{equation*}
\lambda \int_0^1f(v_1(r))\psi(r)\, dr=\int_0^1 w_1(r)\psi'(r)\, dr-\int_0^1\frac{N-1}{r} w_1(r)\psi(r)\, dr\,.
\end{equation*}
That is, the distributional derivative de $w_1$ satisfies
\begin{equation*}
  w'_1=-\lambda f(v_1)-\frac{N-1}{r}w_1\,.
\end{equation*}
As a direct consequence $w'_1\in L^\infty(\sigma,1)$ for all $\sigma>0$ and so condition 7) also holds.

\medskip

\noindent
{\it Proof of 9).}  \quad

Before checking assertion 9), observe that $v_1$ is a function of bounded variation
and $w_1$ satisfies that its derivative is bounded in each interval $(\sigma,1)$. Thus,
the one--dimensional pairing $(w_1,v'_1)$ has sense there.

To see 9), consider $\psi\in C_0^\infty(0,1)$ and define $\varphi\in C_0^\infty(B_1(0))$
as above. It follows from the identity $|D u|=(\z, Du)$ as measures that
\begin{align*}
  \int_{B(0,1)}\varphi|Du| &=\int_{B(0,1)}\varphi(\z, Du)\\
  &= -\int_{B(0,1)}u\, \varphi\, \d\z\ dx-\int_{B(0,1)}u\, \z\cdot\nabla\varphi\ dx\,.
\end{align*}
Performing the same manipulations as above, we obtain
\begin{align*}
  \int_0^{1}\psi|v_1'| &=\lambda\int_0^{1}v_1(t)\psi(r)f(v_1(r))\, dr+\int_0^{1}v_1(r)\psi(r)w_1(r)\left(\frac{N-1}r\right)\, dr
  \\
  &\hspace{7cm}-\int_0^{1}v_1(r)w_1(r)\psi'(r)\, dr\\
  &=-\int_0^{1}v_1(r)\psi(r)w'_1(r)\, dt-\int_0^{1}v_1(r)w_1(r)\psi'(r)\, dr\\
  &=\int_0^{1}\psi(w_1,v'_1)\,,
\end{align*}
as desired.

\medskip

\noindent
{\it Proof of 10).}   \quad

Consider a nonnegative $\psi\in C_0^\infty(0,1)$ and define now
$\varphi\in C_0^\infty(B_1(0))$ by
\begin{equation*}
  \varphi(x)=\left\{ \begin{array}{ll}
  \psi({ |x|}) \frac{N-1}{|x|^{N}}& x \ne 0\,;\\[2mm]
  0 & x=0\,.
  \end{array}\right.
\end{equation*}
Recall that we have proved
$$\int_{B(0,1)}\varphi |Du|=\liminf_{p\to1}\int_{B(0,1)}\varphi |\nabla u(x)|^pdx\,.$$
Considering a further subsequence, if necessary, and passing to polar coordinates, we deduce
\begin{equation*}
  \int_0^{1}\frac{N-1}r\, \psi(r)|v'_1|=\lim_{p\to1}\int_0^{1}\frac{N-1}r\,\psi(r)|v'_p|^pdr \,.
\end{equation*}
Now, note that on the right hand side, we have got the derivative of the functional $E=E_p$ defined in \eqref{energy}. Thus
\begin{align*}
  \int_0^{1}\frac{N-1}r\,\psi(r)|v'_p|^pdr &=\int_0^{1}\psi(r)\left(-\frac{dE_p}{dr}\right)\, dr
    \\
  & =\int_0^{1}\psi'(r) E_p\, dr  \\
  &=\frac1{p'}\int_0^{1}\psi'(r)|w_p(r)|^{p'}\, dr+\lambda_p\int_0^{1}\psi'(r)F(v_p(r))\, dr\,.
\end{align*}
Hence,
\begin{multline}\label{e5:04}
\int_0^{1}\frac{N-1}r\, \psi(r)|v'_1|\\
\hspace{2,8cm} =\lim_{p\to1}\frac1{p'}\int_0^{1}\psi'(r)|w_p(r)|^{p'}\, dr+\lim_{p\to1}\lambda_p\int_0^{1}\psi'(r)F(v_p(r))\, dr\,.
\end{multline}
To compute the first integral on the right hand side, recall again that we have seen the existence of a constant $C>0$ satisfying
\begin{equation*}
  \int_{B(0,1)}|\nabla u_p|^p dx\le C\,,\qquad\hbox{for all }p>1\,.
\end{equation*}
Performing our usual manipulations, we achieve a uniform bound for the family
$
  \int_0^{1}r^{N-1}|w_p(r)|^{p'}dr\,.
$
As a consequence, if $0<a<b<1$ satisfy $\sop(\psi)\subset(a,b)$, then the family
$
\int_a^b|w_p(r)|^{p'}dr
$
is also uniformly estimated, owing to the inequality
$$
  a^{N-1}\int_a^b|w_p(r)|^{p'}dr\le \int_a^br^{N-1}|w_p(r)|^{p'}dr\,.
$$
Therefore, there is certain $C_1>0$ such that
$$
\int_0^{1}|\psi'(r)||w_p(r)|^{p'}\, dr\le C_1\,,\qquad\hbox{for all }p>1\,.
$$
Then, we arrive at
\begin{equation*}
  \lim_{p\to1}\frac1{p'}\int_0^{1}|\psi'(r)||w_p(r)|^{p'}\, dr\le\lim_{p\to1}\frac{p-1}p C_1=0\,.
\end{equation*}
Going back to \eqref{e5:04}, we conclude that
\begin{align*}
  \int_0^{1}\frac{N-1}r\, \psi(r)|v'_1| & =\lim_{p\to1}\lambda_p\int_0^{1}\psi'(r)F(v_p(r))\, dt
  \\
  &=\lambda_1\int_0^{1}\psi'(r)F(v_1(r))\, dr\,.
\end{align*}
Therefore, identity (10) is proved.
\end{proof}

Observe that it follows from conditions (6)--(9) of Proposition \ref{limit1}  that the limit $(v,w,\lambda) = (v_1,w_1,\lambda_1)$ is a solution to the limit system
\begin{equation*}
\left\{
\begin{array}{l}
\frac{v'}{|v'|}=w,
\\[4mm]
\displaystyle w'=-\frac{N-1}r w-\lambda f(v)
\\[5mm]
v(0)=\alpha, \quad w(0)=0,
\end{array}
\right.
\end{equation*}
where $\alpha=\lim_{p\to1}\alpha_p$. Moreover, the extra condition \eqref{clau} holds.
Thus, defining $u(x)=v(|x|)$ and $\z(x)=w(|x|)\frac x{|x|}$, we deduce that $u$ is a radial solution to \eqref{problem} satisfying \eqref{clau}. Therefore, if $(\lambda, u)\in\GG$, then \eqref{clau} holds.
So this condition becomes the key to discerning if a solution to problem \eqref{problem} comes from solutions to p-problems.

\begin{theorem}\label{limit2}
Assume that $N\ge1$.
Radial solutions to problem  \eqref{problem} which satisfy \eqref{clau} are continuous.

As a consequence, for every $0<\lambda<\lambda^*$ there exist exactly two bounded solutions, namely:
\begin{enumerate}
  \item The trivial solution $u(x)=0$.
  \item The constant solution $\displaystyle u(x)=f^{-1}\left(\frac N{\lambda}\right)$
\end{enumerate}
Furthermore, assuming that $N\ge2$, the unbounded solution $\displaystyle u(x)=f^{-1}\left(\frac {N-1}{\lambda |x|}\right)$, which exists for every $0<\lambda\le \overline \lambda$, also satisfies condition \eqref{clau}.
\end{theorem}

\begin{proof} The one--dimensional case is just Corollary \ref{corollary} since any constant solution satisfies condition \eqref{clau}. So, henceforth we consider $N\ge2$.

Fixed $\lambda$, assume that $u(x)=v(|x|)$ is a discontinuous solution to problem \eqref{problem}.
We are checking that it does not satisfy condition \eqref{clau} in the discontinuity set $\{|x|=\rho\}$, with $0<\rho<1$. In this set, condition \eqref{clau} reads as
\begin{align}\label{clau1}
\nonumber  \lambda\left(F(v^+(\rho))-F(v^-(\rho))\right) &=-\frac{N-1}\rho\left|v^+(\rho)-v^-(\rho)\right|\\
  &=\frac{N-1}\rho\left(v^+(\rho)-v^-(\rho)\right)
\end{align}
since $v$ is decreasing.

On account of Theorem \ref{radial}, $v$ is given by
\[v(r)=\left\{
\begin{array}{ll}
\displaystyle f^{-1}\left(\frac{N}{\lambda\rho}\right) &\hbox{if } r<\rho\,;\\[5mm]
\displaystyle f^{-1}\left(\frac{N-1}{\lambda r}\right) &\hbox{if } \rho<r<1\,.
\end{array}\right.
\]
Hence, we have $v^+(\rho)=f^{-1}\left(\frac{N}{\lambda\rho}\right)$ and $v^-(\rho)=f^{-1}\left(\frac{N-1}{\lambda\rho}\right)$. Applying the mean value theorem to $F$ in the interval $[v^-(\rho), v^+(\rho)]$, we find $\xi\in ]v^-(\rho), v^+(\rho)[$ such that
\[F(v^+(\rho))-F(v^-(\rho))=f(\xi)\left(v^+(\rho)-v^-(\rho)\right)\,.\]
Thus, $\displaystyle \frac{N-1}{\lambda \rho}<f(\xi)<\frac{N}{\lambda\rho}$ and so identity \eqref{clau1} does not hold.
\end{proof}


\subsection{Asymptotics for the critical value $\lambda_p^*$ and minimal solutions $w_{\lambda(p)}$  on the unit ball as $p$ approaches $1$.}

\medskip

  In this subsection we compare the critical value $\lambda^*$ and trivial minimal solutions obtained by the $1$--Laplacian problem in the radial case with the limit to the $p$--Laplacian results. Specifically, for $\Omega=B_1(0)$ and $N\geq 1$ we will show  that the critical value $\lambda_p^*$ to problem \eqref{plaplacian} converges to $\lambda^*=\frac{N}{f(0)}$ when $p$ tends to $1$, which is exactly the critical value obtained for problem \eqref{problem} (Theorem \ref{unidimensional} and Theorem \ref{radial}). Furthermore, $w_{\lambda(p)}$, minimal solutions of problem \eqref{plaplacian} tend to trivial solutions.
We recall that these trivial solutions correspond to minimal solutions of problem \eqref{problem} established in Theorem \ref{unidimensional} and Theorem \ref{trivial}, so that the limit of minimal solutions is also a minimal solution. Nevertheless, we are not be able to see that extremal solutions tend to the trivial solution; we only succeed in some specific cases, namely: $f(u)=e^u$ and $f(u)=(1+u)^m$, when $m>e^{-1}$.

  We start by considering the initial value problem \eqref{BVP}:
 \begin{equation}
\left\{
\begin{array}{lc}
r^{1-N}\left(r^{N-1}|u'|^{p-2}u'  \right)'+\lambda f(u)=0,& r \hbox{ in } (0,1),
\\
\\
u>0, & r \hbox{ in } (0,1),
\\
\\
u(0)=\alpha>0, u'(0)=0. &
\end{array}
\right.
\end{equation} Then, there exists a unique solution $(\lambda, u)$ for every $\alpha=\|u\|_\infty>0$ and $\lambda$ can be parameterized in the following way:
\begin{equation}\label{continuum}
\lambda(\alpha)=\alpha^{\,p-1}\left(\int_0^1 \left( t^{1-N}\int_0^t s^{N-1}f(u(s))ds  \right)^{\frac{1}{p-1}}  dt \right)^{1-p}.
\end{equation}
Moreover, the associated bifurcation diagram is a continuum of solutions $(\lambda, u)\in [0,\infty[\times \mathcal{C}([0,1])$, which depends on the dimension $N$.

   In particular, if we take $f(u)=e^u$, the graph of the continuum of solutions is classified into three groups: (1) $N\leq p$, (2) $p<N<\frac{p^2+3p}{p-1}$ and (3) $N\ge \frac{p^2+3p}{p-1}$ (see \cite{JS}). Letting $p$ go to $1$, only the first two cases can be considered. Regarding the first case ($N\leq p$), there are exactly two solutions for each  $\lambda \in (0,\lambda_p^*)$ and one solution for $\lambda = \lambda_p^*$. This behavior of the solutions is reflected in Figure \ref{fig:3} (see below). It is noteworthy to compare it with Figure \ref{fig:1} which corresponds to de limit case $N=p=1$. On the other hand,  if we take into account the second case ( $p<N<\frac{p^2+3p}{p-1}$ ), there is a continuum of solutions which oscillates around the line  $\overline\lambda_{(p)}=p^{\,p-1}(N-p)$ with the amplitude of oscillations tending to zero, as $\|u\|_\infty \to \infty$ as can be seen in Figure \ref{fig:4} below. Observe that the bound $\overline\lambda_{(p)}$ plays a leading role in this second case (for which the problem has infinitely many nontrivial solutions). Something similar happens in the limit case for $p=1$ (Figure \ref{fig:2}). Observe that $\overline\lambda_{(p)} \to N-1$ when $p \to 1$.
It should be noted that for the problem \eqref{problem} there are infinitely many nontrivial solutions with Dirichlet conditions on the boundary (in the sense of the traces) as long as $\overline \lambda=\frac{N-1}{f(0)}$ (Theorem \ref{radial}).
\begin{figure}[H]
\begin{minipage}[b]{0.4\linewidth}
\begin{tikzpicture}[scale=0.50]
\draw[<->] (0,8)--(0,0)--(8,0);
\fill (0.1,8) circle (0.1pt) node[above] {$\|u\|_\infty$};
\fill (8,0) circle (0.1pt) node[right] {$\lambda$};
\fill (0,0) circle (2pt) node[below] {$0$};
\fill (6,0) circle (2pt) node[below] {$\lambda^*_{(p)}$};
\node at (5.5,6) {$N\leq p$};
       \draw[very thick,domain=-0.01:1.06,smooth,variable=\y,blue]  plot ({15.7*\y*exp(-\y)+.2},{\y});
        \draw[very thick,domain=1.06:7.8,smooth,variable=\y,red]  plot ({15.7*\y*exp(-\y)+.2},{\y});
              \draw[dashed]  (6,0)--(6,1.06);
    \end{tikzpicture}
    \caption{}
    \label{fig:3}
    \end{minipage}
\hspace{0.5cm}
%
\begin{minipage}[b]{0.4\linewidth}
\begin{tikzpicture}[scale=0.50]
\draw[<->] (0,8)--(0,0)--(8,0);
\fill (0.1,8) circle (0.1pt) node[above] {$\|u\|_\infty$};
\fill (8,0) circle (0.1pt) node[right] {$\lambda$};
\fill (0,0) circle (2pt) node[below] {$0$};
\fill (6,0) circle (2pt) node[below] {$\lambda^*_{(p)}$};
\fill (3,0) circle (2pt) node[below] {$\overline\lambda_{(p)}$};
\node at (7.8,6) {$p<N<\frac{3p+p^2}{p-1}$};
       \draw[very thick,domain=-0.01:1,smooth,variable=\y,blue]  plot ({15.7*\y*exp(-\y)+.2},{\y});
              \draw[very thick,domain=2:8.9,smooth,variable=\y,red]  plot ({2.95+7*sin(15*\y^2)/\y},{-1+\y});
              \draw[dashed]  (6,0)--(6,1.1);
               \draw[dashed]  (3,0)--(3,8);
    \end{tikzpicture}
    \caption{ }
    \label{fig:4}
    \end{minipage}
\end{figure}
The bifurcation diagrams above shows different features in each case depending on dimension $N$. However, regarding minimal solutions (represented in the previous diagrams in blue)  we will show below that they tend to trivial solutions. We stress that this fact occurs regardless of the dimension $N$ and for any nonlinearity $f$ satisfying hypotheses \eqref{hipotesis}.

\begin{theorem}\label{limit_p-laplacian}
Fix $N\geq 1$, let $p>1$ be small enough and denote by $\{ w_{\lambda(p)} \}_{\lambda \in [0,\lambda_p^*]}$ the increasing branch of positive minimal solutions to problem
\begin{equation}\label{corona}
\left\{
\begin{array}{cc}
-\Delta_p u=\lambda f(u), & \, {\rm{in}} \,\, B_1(0),
\\
u=0, &  \, {\rm{on}} \,\, \partial B_1(0),
\end{array}
\right.
\end{equation}
where $\lambda_p^*$ is the critical value such there is no bounded solutions for $\lambda>\lambda_p^*$, and $f$  satisfies \eqref{hipotesis}. Then,
\begin{enumerate}
\item $\displaystyle \lambda_p^*\to \frac{N}{f(0)}$,\, as $p\to 1$.
\item $\|w_{\tilde \lambda (p)}\|_\infty \to 0$,\, as $p\to 1$ and for every $\tilde \lambda \in [0,\frac{N}{f(0)}[$.
\end{enumerate}
\end{theorem}
\begin{proof}
Firstly, we note that, fixed $N\geq 1$, there exists $0<\delta(N)<1$ such that $N<\frac{p^2+3p}{p-1}$ for every $p\in (1,1+\delta(N))$. This ensures the existence of minimal solutions, $w_{\lambda(p)}$, up to the critical value $\lambda=\lambda_p^*$ (\cite[Theorem 1.3]{CCS}). So from now on we assume  that $1<p<1+\delta(N)$.

  {\underline {Proof of $\emph{(1)}$}:} To begin with, observe that by \eqref{continuum} we obtain
\begin{align*}
\nonumber  \lambda_p^* & \geq \lambda (\alpha)
\\
\nonumber  &= \alpha^{\,p-1}\left(\int_0^1 \left( t^{1-N}\int_0^t s^{N-1}f(u(s))ds  \right)^{\frac{1}{p-1}}  dt \right)^{1-p}
\\
\nonumber &\geq  \alpha^{\,p-1}\left(\int_0^1 \left( t^{1-N}\int_0^t s^{N-1}f(\alpha) ds  \right)^{\frac{1}{p-1}}  dt \right)^{1-p}
\\
\nonumber & =N\left(\frac{p-1}{p}  \right)^{1-p}\frac{\alpha^{p-1}}{f(\alpha)}\,.
\end{align*}
Since this inequality holds for all positive $\alpha$, we obtain
\begin{equation}\label{underbound}
\lambda_p^*\geq N\left(\frac{p}{p-1}  \right)^{p-1} \max_{\alpha \in [0,\infty[}\frac{\alpha^{p-1}}{f(\alpha)}.
\end{equation}
Denote by $F_p(\alpha):=\frac{\alpha^{p-1}}{f(\alpha)}$. Obviously, $F_p\in \mathcal{C}^1$ is nonnegative with $F_p(0)=0$ and $F_p(\alpha)\to 0$ when $\alpha \to \infty$ (since $f^{1/p-1}(s)$ is superlinear). Then, $F_p$  has its maximum in some $\overline \alpha_p \in ]0,\infty[$ (i.e., $F_p(\overline \alpha_p)=\max_{\alpha \in [0,\infty[}F_p(\alpha))$. Moreover, $F^{\, \prime}(\overline \alpha_p)=0$ implies
\begin{equation}\label{vamos}
\frac{\overline \alpha_pf^{\,\prime}(\overline \alpha_p)}{f(\overline \alpha_p)}=p-1
\end{equation}
Now, we claim the sequence $\{\overline \alpha_p\}_p$ is bounded. Looking for a contradiction, we assume that there exists a (not relabeled) subsequence $\overline \alpha_p \to \infty$ as $p \to 1$, and we now show that this fact is in contradiction with \eqref{vamos}. First, observe that
\begin{equation}\label{boda}
\lim_{s\to \infty} \inf \frac{sf^{\,\prime}(s)}{f(s)}=\gamma>0
\end{equation}
otherwise there exists $s_0>0$ such that
$$
\frac{sf^{\,\prime}(s)}{f(s)}<\varepsilon, \,\, \hbox{ for all }\, s\geq s_0,
$$
holds for every $0<\varepsilon<\min \{p-1,\gamma\}$.
This implies that $\displaystyle \frac{f(s)}{s^\varepsilon}$ is decreasing for $s\geq s_0$ since
\begin{align*}
\left( \frac{f(s)}{s^\varepsilon} \right)^{\, \prime} & =\frac{f^{\, \prime}(s)s^\varepsilon-f(s)\varepsilon s^{\varepsilon-1}}{s^{2\varepsilon}}
\\
& =\frac{f(s)}{s^{\varepsilon +1}}\left(\frac{sf^{\, \prime}(s)}{f(s)} -\varepsilon \right)<0,
\end{align*}
 for all  $s\geq s_0$. The fact that function $\displaystyle s \mapsto \frac{f(s)}{s^\varepsilon}$ is decreasing for $s\geq s_0$
is in contradiction with the fact that $f^{1/p-1}$ is superlinear (because $\varepsilon < p-1$). Then, if $\overline \alpha_p \to \infty$ as $p \to 1$ and  by using \eqref{boda} and \eqref{vamos} we obtain the following contradiction
$$
0<\gamma  =\lim_{s\to \infty} \inf \frac{sf^{\,\prime}(s)}{f(s)} \leq \lim_{p\to 1}\frac{\overline \alpha_pf^{\,\prime}(\overline \alpha_p)}{f(\overline \alpha_p)}=\lim_{p\to 1}(p-1)=0.
$$
Hence, we have checked that the sequence $\{\overline \alpha_p\}_p$ is bounded.
Thus, taking limits in $F_p(\alpha)\leq F_p(\overline \alpha_p)$ when $p\to 1$, we obtain that $F_1(\alpha)\leq F_1(\overline \alpha)$ for all $\alpha\geq 0$, being
$$
F_1(\alpha)=\left\{ \begin{array}{lc} 0,  &  \hbox { if } \alpha=0, \\ \frac{1}{f(\alpha)},  &  \hbox { if } \alpha>0. \end{array}  \right.
$$
Therefore,  $ F_1(\overline \alpha)=\sup_{\alpha\geq 0}F_1(\alpha)=\frac{1}{f(0)}$. As a consequence,
$$
\lim_{p\to 1} \max_{\alpha \in [0,\infty[}\frac{\alpha^{p-1}}{f(\alpha)}=\lim_{p\to 1}F_p(\overline \alpha_p)= F_1(\overline \alpha) =    \frac{1}{f(0)}.
$$
Thus, by \eqref{underbound}, we obtain the following lower bound
\begin{equation}\label{lower_bound}
\lambda_p^* \geq N\left(\frac{p}{p-1}  \right)^{p-1}F_p(\overline \alpha_p)\to \frac{N}{f(0)}, \,\,\,\,\, p\to 1.
\end{equation}
On the other hand, in order to establish an upper bound to $\lambda_p^*$, we take into account  the following inequality from the proof of \cite[Theorem 1.4]{CS1}
\begin{equation}\label{Cabre-S}
\lambda_p^*\leq \max \left\{\lambda_{1(p)},\lambda_{1(p)}F_p(\overline \alpha_p)   \right\},
\end{equation}
where $\lambda_{1(p)}$ is the principal eigenvalue of the $p$--Laplacian in $\Omega=B_1(0)$.
We next prove that it actually holds
\begin{equation*}
  \lambda_p^*\leq \lambda_{1(p)}F_p(\overline \alpha_p)\,.
\end{equation*}
To this end, we fix $p$, denote by $w_{\lambda(p)}$ the minimal solution to \eqref{corona} and take $K>F_p(\overline \alpha_p)^{1/(1-p)}$, so it follows that $u_{\lambda(p)}:=Kw_{\lambda(p)}$ is solution to
\begin{equation}\label{confinado}
\left\{
\begin{array}{lc}
-\Delta_p u_{\lambda(p)}=\lambda \, K^{p-1}\tilde f(u_{\lambda(p)}), & \, {\rm{in}} \,\, B_1(0),
\\
u_{\lambda(p)}=0, &  \, {\rm{on}} \,\, \partial B_1(0),
\end{array}
\right.
\end{equation}
where $\tilde f(s)=f\left(\frac{s}{K}  \right)$. Note that $\tilde f$ is also under the hypotheses of \eqref{hipotesis}. It follows, by the change of variable $\alpha=Kt$, that
$$
\tilde F_p(\overline \alpha_p):=\max_{\alpha \in [0,\infty[}\frac{\alpha^{p-1}}{\tilde f(\alpha)}=\max_{t \in [0,\infty[}\frac{(tK)^{p-1}}{f(t)}=K^{p-1}F_p(\overline \alpha_p).
$$
Finally, using inequality \eqref{Cabre-S} in \eqref{confinado} and the last equality, we get
$$
\lambda_p^*\, K^{p-1}\leq \max \left\{\lambda_{1(p)},\lambda_{1(p)}\tilde F_p(\overline \alpha_p)    \right\}=\max \left\{\lambda_{1(p)},\lambda_{1(p)} K^{p-1}F_p(\overline \alpha_p)    \right\}
$$
and then,
$$
\lambda_p^* \leq \max \left\{\frac{\lambda_{1(p)}}{ K^{p-1}},\lambda_{1(p)}\,F_p(\overline \alpha_p)    \right\}.
$$
Taking into account that it holds for any $K>F_p(\overline \alpha_p)^{1/(1-p)}$, we establish the following upper bound of $\lambda_p^*$
\begin{equation}\label{upper_bound}
\lambda_p^* \leq \lambda_{1(p)}\,F_p(\overline \alpha_p)\to \frac{N}{f(0)}, \,\,\,\,\, p\to 1.
\end{equation}
Where we have used that $\lambda_{1(p)} \to h(B_1(0))=N$ when $p\to 1$, being $h(B_1(0))$ the Cheeger constant for the unit ball  (\cite{KF}).


{\underline {Proof of $\emph{(2)}$}:} Fix $0<\tilde \lambda <\frac{N}{f(0)}$. Due to the previous lower bound \eqref{lower_bound}, there exists $p_0>1$ small enough such that $\tilde \lambda < \lambda_{p}^*$ for every $p\in ]1,p_0[$. This ensures the existence of minimal solutions, $w_{\tilde \lambda (p)}$, to \eqref{corona} with $\lambda=\tilde \lambda$ and $1<p<p_0$.

   We argue by contradiction and suppose there exists a sequence $\{p_n\}\subset ]1,p_0]$, with $p_n \to 1$, such that $\|w_{\tilde \lambda (p_n)}\|_\infty  \to \beta \in ]0,\infty]$ when $n\to \infty$. Now, we fix $\kappa>0$ satisfying
   \begin{equation}\label{betax}
   \kappa<\min \left\{ \beta, f^{-1}\left(N/\tilde \lambda \right)  \right\}
   \end{equation}
Since the branch of minimal solutions to \eqref{corona} is positive and increasing with respect to $\lambda$, there exists a sequence $\{\lambda_n\}_n$ with $0<\lambda_n\leq \tilde \lambda $ such that
   \begin{equation}\label{virus}
   \|w_{\lambda_n (p_n)}\|_\infty=\kappa, \quad \hbox{ for every } n\geq n_0.
   \end{equation}

     On the other hand, by \eqref{gador} we get
     \begin{align*}
     \|w_{\lambda_n (p_n)}\|_\infty& =\int_0^1 \left[ \frac{\lambda_n}{t^{N-1}}\int_0^t s^{N-1}f(w_{\lambda_n (p_n)}(s))ds  \right]^{\frac{1}{p_n-1}}dt
\\ &\leq f\left( \|w_{\lambda_n (p_n)}\|_\infty  \right)^{\frac{1}{p_n-1}}\int_0^1 \left[ \frac{\lambda_n}{t^{N-1}}\int_0^t s^{N-1}ds  \right]^{\frac{1}{p_n-1}}dt
\\ &=\left(\frac{\lambda_n  f\left( \|w_{\lambda_n (p_n)}\|_\infty  \right)}{N}    \right)^{\frac{1}{p_n-1}}\frac{p_n-1}{p_n}.
     \end{align*}

Replacing this inequality by \eqref{virus} and taking in account  \eqref{betax} and that $\lambda_n\leq \tilde \lambda$, it follows that
$$
\kappa < \left( \frac{\lambda_n}{\tilde \lambda} \right)^{\frac{1}{p_n-1}}\frac{p_n-1}{p_n}\to 0, \quad \hbox{ as } p_n \to 1,
$$
 which is a contradiction since we had fixed $\kappa>0$.
\end{proof}

Note that inequalities \eqref{lower_bound} and \eqref{upper_bound} provide us with $\lambda_p^*$ estimates for specific nonlinearities. In particular, in the following corollary we provide such estimates for the typical   exponential and potential types nonlinearities.

\begin{corollary}\label{good_estimates} Let $p>1$ and $N\geq 1$. Consider $\lambda_p^*(f)$ the critical value to problem \eqref{corona} for $f(u)=e^u$ or $f(u)=(1+u)^m$ with $m>p-1$ . Then, the following estimates holds
\begin{enumerate}
\item$$N\left(\frac{p}{e}  \right)^{p-1}\leq\lambda_p^*(e^u)\leq N\left(\frac{p}{e}  \right)^{p-1}\frac{\Gamma \left(p+1+\frac{N(p-1)}{p}\right)}{\Gamma(p+1)\Gamma \left(2+\frac{N(p-1)}{p}\right)}
$$
\item
$$
 \hspace{-2.5cm}    \frac{Np^{\,p-1}(m-p+1)^{m-p+1}}{m^m}  \leq\lambda_p^*((1+u)^m)
$$

$$
\hspace{4cm} \leq \frac{Np^{\,p-1}(m-p+1)^{m-p+1}}{m^m}\frac{\Gamma \left(p+1+\frac{N(p-1)}{p}\right)}{\Gamma(p+1)\Gamma \left(2+\frac{N(p-1)}{p}\right)}
$$
\end{enumerate}
where $\Gamma(z)=\int_0^\infty t^{z-1}e^{-t}dt$ is the Gamma function.
\end{corollary}
\begin{proof}
In \cite{BD} the authors give the estimate from above for the first eigenvalue of the $p$--Laplacian operator on the unit ball:
$$
\lambda_{1(p)}\leq N\left(\frac{p}{p-1}  \right)^{p-1}\frac{\Gamma \left(p+1+\frac{N(p-1)}{p}\right)}{\Gamma(p+1)\Gamma \left(2+\frac{N(p-1)}{p}\right)}.
$$
In order to prove \emph{(1)}, it is enough to observe that
$$
F_p(\overline \alpha_p)=\max_{\alpha \in [0,\infty[}\frac{\alpha^{p-1}}{e^{\alpha}}=\left(\frac{p-1}{e}  \right)^{p-1},
$$
where replacing in \eqref{lower_bound} and \eqref{upper_bound} we get the desired inequalities. Analogously, to prove \emph{(2)} we use that
$$
F_p(\overline \alpha_p)=\max_{\alpha \in [0,\infty[}\frac{\alpha^{p-1}}{(1+\alpha)^m}=\frac{(p-1)^{\,p-1}(m-p+1)^{m-p+1}}{m^m}.
$$
\end{proof}

  In the following result we establish that, for certain reaction terms $f(u)$, extremal solutions $u_p^*$ to \eqref{corona} tend to zero.

\begin{proposition} Let $\Omega=B_1(0)$  ($N\geq 1$) and let $u_p^*$ be the solution to \eqref{corona}  with $\lambda=\lambda_p^*$ for $f(u)=e^u$ or for $f(u)=(1+u)^m$ with $m>e^{-1}$. Then, $\|u_p^*\|_\infty \to 0$ when $p \to 1$.
\end{proposition}
\begin{proof}
In case $f(u)=e^u$, by Corollary \ref{good_estimates}, $\lambda_p^*(e^u)\leq N \left(\frac{p}{e}  \right)^{p-1}\, G(p,N)$
where $G(p,N):= \frac{\Gamma \left(p+1+\frac{N(p-1)}{p}\right)}{\Gamma(p+1)\Gamma \left(2+\frac{N(p-1)}{p}\right)}$. Observe that for $N\geq1$ we obtain $G(1,N)=1$. Moreover, we recall that $\Gamma^{\, \prime}(z)=\Gamma(z)\psi(z)$ where $\psi(z)=\int_0^\infty \left(\frac{e^{-t}}{t}-\frac{e^{-zt}}{1-e^{-t}}   \right)dt$ is the Digamma function which satisfies $\psi(2)=1-\gamma$ being $\gamma$ the Euler-Mascheroni constant. Then, we get that
\begin{multline*}
\frac{\partial G(p,N)}{\partial p}=
\\
\frac{\Gamma \left(p+1+\frac{N(p-1)}{p}\right) \psi \left(p+1+\frac{N(p-1)}{p} \right)\left(1+\frac{N}{p^2}  \right)\Gamma(p+1)\Gamma \left(2+\frac{N(p-1)}{p}\right) }{\Gamma(p+1)^2\,\Gamma \left(2+\frac{N(p-1)}{p}\right)^2}
\\
-\frac{\Gamma \left(p+1+\frac{N(p-1)}{p}\right) \Gamma(p+1)\psi(p+1)\Gamma \left(2+\frac{N(p-1)}{p}\right) }{\Gamma(p+1)^2\,\Gamma \left(2+\frac{N(p-1)}{p}\right)^2}
\\
-\frac{\Gamma \left(p+1+\frac{N(p-1)}{p}\right) \Gamma(p+1)\Gamma \left(2+\frac{N(p-1)}{p}\right)\psi  \left(2+\frac{N(p-1)}{p}\right)\frac{N}{p^2} }{\Gamma(p+1)^2\,\Gamma \left(2+\frac{N(p-1)}{p}\right)^2}.
\end{multline*}
As a consequence,
\begin{align*}
\lim_{p\to 1^+}\frac{\partial G(p,N)}{\partial p} &= (1-\gamma)(1+N)-(1-\gamma)-(1-\gamma)N=0
\end{align*}
where have we used that $\Gamma(2)=1$. Therefore
\begin{multline*}
\lim_{p\to 1^+}\frac{\partial}{\partial p}\left(\left(\frac{p}{e}  \right)^{p-1}G(p,N)\right) =
\\
 \lim_{p\to 1^+} \left[\left(\frac{p}{e}  \right)^{p-1}\left(\log p-\frac{1}{p} \right)G(p,N)+ \left(\frac{p}{e}  \right)^{p-1} \frac{\partial G(p,N)}{\partial p}   \right]=-1.
\end{multline*}
In conclusion, the function $p\mapsto \left(\frac{p}{e}  \right)^{p-1}G(p,N)$ is decreasing and less than $1$ in a neighborhood of $p = 1$. Thus, there is $\delta>0$ such that
$$
\lambda_p^*(e^u)\leq N \left(\frac{p}{e}  \right)^{p-1}\, G(p,N)<N, \qquad 1<p<1+\delta.
$$
Therefore, if we take as $\tilde \lambda =\lambda_p^*$ into the proof of $(2)$ from Theorem \ref{limit_p-laplacian} and argue like there, then we will prove that $\|w_{\lambda_p^*}\|_\infty=\|u_p^*\|_\infty \to 0$ as $p\to 1$.

  Similarly, taking $f(u)=(1+u)^m$, by Corollary \ref{good_estimates},
   $$\lambda_p^*((1+u)^m)\leq N \frac{p^{\,p-1}(m-p+1)^{m-p+1}}{m^m}\, G(p,N).$$
Note that in this case
\begin{multline*}
\lim_{p\to 1^+}\frac{\partial}{\partial p}\left(  \frac{p^{\,p-1}(m-p+1)^{m-p+1}}{m^m}  G(p,N)\right) =
\\
\lim_{p\to 1^+} \frac{p^{\,p-1}(m-p+1)^{m-p+1}}{m^m}\left[ \left(\log \frac{p}{m-p+1}-\frac{1}{p}    \right)G(p,N)+   \frac{\partial G(p,N)}{\partial p}\right]
\\
=-\log m-1<0,
\end{multline*}
since $m>e^{-1}$. Thus,
$$
\lambda_p^*((1+u)^m)<N, \qquad 1<p<1+\delta.
$$
for some $\delta>0$. Finally, reasoning as before, we arrive at the desired result.
\end{proof}
\begin{remark}\rm
Note that, when proving that the extremal solution $u_p^*$ tends to zero, in our method it is essential that
$\lambda_p^*<\frac{N}{f(0)}$ for $p$ small enough close to $1$. Unfortunately this method is not general. For instance, if we choose $f(u)=1+u^2$, then it is not possible to show that $\lambda_p^*<N$ since now our upper bound \eqref{upper_bound} is increasing
 near $p=1$.
\end{remark}

\section*{Ackonwledgement}
The authors wish to thank Prof. J. Carmona for valuable comments concerning this paper.

The first author is partially supported by PGC2018-096422-B-I00 \newline (MCIU/AEI/FEDER, UE) and by Junta de Andaluc\'ia FQM-116 (Spain). The second author is supported by the Spanish Ministerio de Ciencia, Innovaci\'on y Universidades and FEDER, under project PGC2018--094775--B--I00.

\end{document}